\documentclass[11pt,a4paper,fleqn]{article}
\usepackage{amsthm,amsmath,amssymb,amsfonts,hyperref,color}
\newtheorem{theorem}{Theorem}

\newtheorem{example}{Example}

\numberwithin{equation}{section}
\numberwithin{lemma}{section}
\numberwithin{theorem}{section}
\numberwithin{corollary}{section}

\allowdisplaybreaks

\begin{document}
	\title{On the dual valued generalized hypergeometric function and its special cases}
	\author{Ravi Dwivedi$^{1,}$\footnote{E-mail: dwivedir999@gmail.com}  
		\,
		 and Juan Carlos Cort\'es$^{2,}$\footnote{E-mail: jccortes@mat.upv.es (Corresponding author)}
		\\ ${}^{1}$Department of Science, Swami Atmanand Govt. Engish \\
		Medium Model College Jagdalpur, Bastar (CG) 494001, India.\\
	${}^{2}$ Instituto de Matem\'atica
	Multidisciplinar, Universitat\\
	Polit\'ecnica 
	de Val\'encia, 46022
	Val\'encia, Spain}
	\maketitle
	\begin{abstract}
This paper explores the calculus of dual-valued functions and investigates the gamma function, beta function and generalized hypergeometric functions by incorporating dual numbers as parameters and variables. We examine its fundamental properties, including regions of convergence, differential equations, and integral representations. 
Furthermore, we provide an in-depth discussion on the various properties of the dual confluent and Gauss hypergeometric functions.
 
		\medskip
		\noindent\textbf{Keywords}: Dual numbers, Dual functions, Dual Gamma and Beta Functions, Confluent hypergeometric function, Gauss hypergeometric function, Generalized hypergeometric function.
		
		\medskip
		\noindent\textbf{AMS Subject Classification}: 30G35, 33C05, 33C15, 33C20. 
	\end{abstract}
	
	\section{Introduction}\label{s:i}
	The theory of dual numbers-expressions of the form $\hat{x} = x + \epsilon x'$, where $\epsilon^2 = 0$ has gained considerable attention for its applications in \textit{automatic differentiation}, \textit{geometric modeling}, and \textit{symbolic computation}~\cite{gw, maclane1998categories, stitzinger1982dual}. In recent years, the incorporation of dual numbers into the framework of \textit{special functions}, notably the \textit{hypergeometric functions}, has led to the development of \textit{dual hypergeometric functions}, which encapsulate both function values and first-order sensitivity data within a single algebraic structure.
	
	Hypergeometric functions, such as the generalized function ${}_pF_q$, arise naturally as solutions to a wide class of linear differential equations, including those found in quantum mechanics, fluid dynamics, and combinatorics~\cite{AAR,emo}. Extending these functions to dual arguments enables the \textit{simultaneous computation of function values and their derivatives} with respect to either their parameters or their main argument. This dual framework aligns closely with the principles of \textit{forward-mode automatic differentiation}~\cite{giles2008extended}, allowing derivatives to be computed analytically and with greater numerical stability compared to finite difference methods.
	
	Moreover, dual hypergeometric functions offer significant advantages in \textit{perturbation theory} and \textit{asymptotic analysis}, where sensitivity to parameter changes is often of primary interest. By encoding infinitesimal variations directly into dual parameters-such as evaluating ${}_2F_1(a + \epsilon b, c; d; z)$ one obtains both the function and its directional derivative in a compact form. This mechanism is particularly effective in \textit{parametric sensitivity analysis}, where hypergeometric terms are involved in models with tunable parameters, such as in statistical distributions and Bayesian priors~\cite{andrews1974scale, devroye1986non,fox1955moment}.
	
	In symbolic computation systems, dual-number extensions of special functions facilitate \textit{derivative-aware symbolic manipulation}, enabling closed-form simplification and symbolic differentiation pipelines that are otherwise cumbersome when relying solely on external derivative formulas. This is beneficial in modern \textit{computer algebra systems} and \textit{machine learning frameworks} that increasingly rely on symbolic-numeric hybrid processing~\cite{abadi2016tensorflow, davenport1993computer}.
	
	Recent studies have also suggested that dual versions of special functions can be valuable in solving \textit{parameterized differential equations} and constructing \textit{approximate solutions} in nonlinear regimes~\cite{bender1978advanced, hormander1983analysis, olver1993applications}. By applying dual inputs to known solutions, one can extract higher-order behaviors or local linearizations without re-deriving the underlying solution structure from scratch.
	
	In this work, we systematically develop the theory of \textit{dual hypergeometric functions}, deriving general expansion rules and proving algebraic identities where their use may offer both theoretical and practical benefits. 
\section{Basic definitions and preliminaries}
Throughout this paper, the set of real numbers and dual real numbers will be denoted by $\mathbb{R}$ and $\mathbb{D} \mathbb{R}$, respectively. A number of the form $\widehat{x} = x_1 + \varepsilon x_2, x_1, x_2 \in \mathbb{R}$ is called a dual real number or simply a dual number,  where $\varepsilon$ is an infinitesimal symbol  such that $\varepsilon^2 = 0$. The real numbers $x_1$ and $x_2$ are called real and dual part of $\widehat{x}$, respectively. We denote the real part of a dual number $\widehat{x}$ by $\Re(\widehat{x})$. Let $\widehat{x} = x_1 + \varepsilon x_2$, $\widehat{y} = y_1 + \varepsilon y_2 \in \mathbb{D}\mathbb{R}$, for all $x_1, x_2, y_1, y_2 \in \mathbb{R}$. Then, 
they can be added or subtracted in the following way:
\begin{align}
	\widehat{x} \pm \widehat{y} = (x_1 \pm y_1) + \varepsilon \, (x_2 \pm y_2).
\end{align}
They can be multiplied as given below
\begin{align}
	\widehat{x} \cdot  \widehat{y} =   x_1 \, y_1 + \varepsilon (x_1 \, y_2 + y_1 \, x_2).\label{2.2}
\end{align}
The inverse of a dual number $\widehat{x} = x_1 + \varepsilon x_2$ can be obtained by
\begin{align}
	\frac{1}{\widehat{x}}  = \frac{1}{x_1} - \varepsilon \frac{x_2}{x_1^2}, \quad x_1 \ne 0.
\end{align}
Further, the quotient of two dual numbers can be derived as
\begin{align}
	\frac{\widehat{x}}{\widehat{y}} = \frac{x_1 + \varepsilon x_2}{y_1 + \varepsilon y_2} = \frac{x_1 \, y_1 + \varepsilon (x_2 \, y_1 - x_1 \, y_2)}{y_1^2}, \quad y_1 \ne 0. \label{e2.4}
\end{align}
The absolute value of a dual number $\widehat{x}$ is given by $\vert \widehat{x}\vert = \vert x_1 + \varepsilon x_2\vert = \vert \Re (\widehat{x})\vert = \vert x_1\vert$ and, by \eqref{2.2} $\vert\widehat{x} \widehat{y}\vert = \vert\widehat{x}\vert \, \vert \widehat{y}\vert$. 

The exponent of a dual number, say $\widehat{b}$, with base a real number can be obtained as follows
\begin{align}
	a^{\widehat{b}} & = a^{b_1 + \varepsilon \, b_2} = a^{b_1} (1 + \varepsilon \, b_2 \log a), \quad a > 0.\label{2.4}
\end{align}
Indeed, observe that
\begin{align}
a^{b_1 + \varepsilon \, b_2} = a^{b_1} \, a^{\varepsilon \, b_2} = a^{b_1} \, e^{\varepsilon b_2 \, \log (a)}. \label{2.5}
\end{align}
The Taylor expression of the last function is 
\begin{align}
e^{\varepsilon b_2 \, \log (a)} = 1 + {\varepsilon b_2 \, \log (a)}, \label{2.6}
\end{align}
Since $\varepsilon^2 = 0$ and higher-order terms vanish. From \eqref{2.5} and \eqref{2.6}, one obtains \eqref{2.4}. 

The rule \eqref{2.4} can be generalized for two dual numbers 
\begin{align} 
	(\widehat{a})^{\widehat{b}} & = (a_1 + \varepsilon \, a_2)^{b_1 + \varepsilon \, b_2} = (a_1)^{b_1} \left(1 + \varepsilon \left(\frac{a_2 \, b_1}{a_1} +  b_2 \log a_1\right)\right), \quad a_1 > 0.\label{2.7}
\end{align} 
 To justify this expression,observe that
 \begin{align}
 (\widehat{a})^{\widehat{b}} = e^{\widehat{b} \, \log (\widehat{a})}. \label{e2.8}
 \end{align}
Now, utilizing the Taylor expansion
\begin{align}
\log (\widehat{a}) = \log (a_1 + \varepsilon \, a_2) = \log (a_1) + \varepsilon \, \frac{a_2}{a_1}. \label{e2.9}
\end{align}
So, 
\begin{align}
\widehat{b} \, \log (\widehat{a}) = (b_1 + \varepsilon \, b_2) \, \left(\log (a_1) + \varepsilon \, \frac{a_2}{a_1}\right) = b_1 \, \log (a_1) + \varepsilon \left(b_2 \, \log (a_1) + b_1 \, \frac{a_2}{a_1}\right),\label{e2.10}
\end{align}
where, we have utilized that $\varepsilon^2 = 0$. Substituting \eqref{e2.10} into \eqref{e2.8}, one gets
\begin{align}
(\widehat{a})^{\widehat{b}} & = e^{b_1 \, \log (a_1)} \, e^{\varepsilon \left(b_2 \, \log (a_1) + b_1 \, \frac{a_2}{a_1}\right)} \nonumber\\
& = e^{b_1 \, \log (a_1)} \, \left(1 + \varepsilon \left(b_2 \, \log (a_1) + b_1 \, \frac{a_2}{a_1}\right)\right)\nonumber\\
& = (a_1)^{(b_1)} \, \left(1 + \varepsilon \left(b_2 \, \log (a_1) + b_1 \, \frac{a_2}{a_1}\right)\right), \quad a_1 > 0,
\end{align}
utilizing the Taylor expansion of the exponential function. Finally, the following particular case of \eqref{2.7} can be obtained as
\begin{align}
	(\widehat{a})^{{b}}  = (a_1)^{(b)} \, \left(1 + \varepsilon \left(b \, \frac{a_2}{a_1}\right)\right), \quad a_1 \ne 0.
\end{align}
Let $f$ be a real-valued differential function. Then, for a dual number $\widehat{x} = x_1 + \varepsilon x_2$, the dual function $f(\widehat{x})$ is given by 
\begin{align*}
f(\widehat{x}) = f(x_1) + \varepsilon x_2 \, f'(x_1),
\end{align*}
where $f'(x_1)$ is the ordinary derivative of $f(x)$ at $x = x_1$. Notice that this definition is motivated by the first-order truncation of Taylor expansion of $f(\widehat{x})$.
\section{The calculus of dual valued functions}
It can be clearly observed that, the dual valued function $f(\widehat{x})$ depends on the real number $x_1$ and, hence, for a continuous differentiable function $f(x)$, we define the dual derivative of $f(\widehat{x})$  by
\begin{align}
	f'(\widehat{x}) & = \frac{d}{d \widehat{x}} f(\widehat{x})  = \lim_{\Delta \widehat{x} \to 0} \frac{f(\widehat{x} + \Delta \widehat{x}) - f(\widehat{x})}{\Delta \widehat{x}}\nonumber\\
	& = \lim_{\Delta {x_1} \to 0} \frac{f({x_1} + \Delta {x_1}) - f({x_1})}{\Delta {x_1}} + \varepsilon \, x_2 \lim_{\Delta {x_1} \to 0} \frac{f'({x_1} + \Delta {x_1}) - f'({x_1})}{\Delta {x_1}} \nonumber\\
	& = f'(x_1) + \varepsilon x_2 f''(x_1). 
\end{align} 
Which can also be written as
\begin{align}
\frac{d}{d \widehat{x}} f(\widehat{x}) = \frac{\partial}{\partial x_1} \left(f(x_1) + \varepsilon x_2 f'(x_1)\right).
\end{align}
\begin{example}
For a dual number $\widehat{x} = x_1 + \varepsilon x_2$, if a dual valued function $f(\widehat{x})$ is defined by $f(\widehat{x}) = (\widehat{x})^k$, then the derivative of $f(\widehat{x})$ can be deduced by
\begin{align*}
	\frac{d}{d \widehat{x}} \, f(\widehat{x}) & = \frac{d}{d \widehat{x}} \, (\widehat{x})^k\\
	& = \frac{\partial}{\partial x_1} \, (x_1^k + k \, \varepsilon x_2 \, x_1^{k - 1} )\\
	& = k \, x_1^{k - 1} + k \, (k - 1) \, \varepsilon x_2 \, x_1^{k - 2} = k (\widehat{x})^{k - 1}.
\end{align*}
\end{example}
Similarly, the following dual derivative formulas can be derived:
\begin{align*}
	& \frac{d}{d \widehat{x}} \, e^{\widehat{x}} = e^{\widehat{x}}; \\
	&  \frac{d}{d \widehat{x}} \,  \sin (\widehat{x}) = \cos (\widehat{x});\\
	 & \frac{d}{d \widehat{x}} \,  \cos (\widehat{x}) = - \sin (\widehat{x}); \\
	 &  \frac{d}{d \widehat{x}} \,  \tan (\widehat{x}) = \sec^2 (\widehat{x});\\
	 & \frac{d}{d \widehat{x}} \,  \cot (\widehat{x}) = - \csc^2 (\widehat{x})\\
	 & \frac{d}{d \widehat{x}} \,  \sec (\widehat{x}) = \sec (\widehat{x}) \, \tan (\widehat{x})\\
	 & \frac{d}{d \widehat{x}} \,  \csc (\widehat{x}) = - \csc (\widehat{x}) \, \cot (\widehat{x})\\
	 & \frac{d}{d \widehat{x}} \,  \log (\widehat{x}) = \frac{1}{\widehat{x}}, \quad x_1 > 0.
\end{align*}
\begin{theorem}
For dual valued functions $f(\widehat{x})$ and $g(\widehat{x})$, the following algebraic properties for a dual derivative operator can be verified easily
\begin{align}
 \frac{d}{d \widehat{x}} (c_1 \, f(\widehat{x}) \pm c_2 \, g(\widehat{x})) & = c_1 \, \frac{d}{d \widehat{x}} f(\widehat{x}) \pm c_2 \, \frac{d}{d \widehat{x}} g(\widehat{x}), \quad c_1, c_2 \in \mathbb{R};\label{3.3}
 \\[5pt]
 \frac{d}{d \widehat{x}} ( f(\widehat{x}) \cdot  g(\widehat{x})) & =  \left(\frac{d}{d \widehat{x}} f(\widehat{x})\right) g(\widehat{x}) + f(\widehat{x}) \,  \frac{d}{d \widehat{x}} g(\widehat{x});
 \\[5pt]
 \frac{d}{d \widehat{x}} \left(\frac{f(\widehat{x})} {g(\widehat{x})}\right) & =  \frac{\left(\frac{d}{d \widehat{x}} f(\widehat{x})\right) g(\widehat{x}) - f(\widehat{x}) \,  \frac{d}{d \widehat{x}} g(\widehat{x})}{[g(\widehat{x})]^2}, \quad g(x_1) \ne 0;
 \\[5pt]
 \frac{d}{d \widehat{x}}  f(g(\widehat{x})) & =    f'(g(\widehat{x})) \,  \frac{d}{d \widehat{x}} g(\widehat{x}). 
\end{align}
\end{theorem}
\begin{proof}
	Since $f(\widehat{x})$ and $g(\widehat{x})$ are dual valued functions, so we write
	\begin{align*}
	f(\widehat{x}) = f(x_1) + \varepsilon \, x_2 \, f'(x_1), \quad g(\widehat{x}) = g(x_1) + \varepsilon \, x_2 \,  g'(x_1),
	\end{align*}
and 
\begin{align*}
	\frac{d}{d \widehat{x}} (c_1 \, f(\widehat{x}) \pm c_2 \, g(\widehat{x})) & = \frac{\partial}{\partial x_1} (c_1 \, f( {x_1}) \pm c_2 \, g( {x_1})) + \varepsilon \, x_2 \, \frac{\partial}{\partial x_1} (c_1 \, f'( {x_1}) \pm c_2 \, g'( {x_1})).
\end{align*}
The partial derivatives are linear and this gives 
\begin{align*}
	\frac{d}{d \widehat{x}} (c_1 \, f(\widehat{x}) \pm c_2 \, g(\widehat{x})) & = c_1 \, f'( {x_1}) \pm c_2 \, g'( {x_1}) + \varepsilon \, x_2 \, ( c_1 \, f''( {x_1}) \pm  c_2 \, \varepsilon \, x_2 \, g''( {x_1}))\nonumber\\
	& = c_1   (f'(x_1) + \varepsilon \, x_2 \, f'' (x_1)) \pm c_2  (g'(x_1) + \varepsilon \, x_2 \, g'' (x_1))\nonumber\\
	& = c_1 \frac{\partial}{\partial x_1} (f(x_1) + \varepsilon \, x_2 \, f' (x_1)) \pm c_2 \frac{\partial}{\partial x_1} (g(x_1) + \varepsilon \, x_2 \, g' (x_1))\nonumber\\
	& = c_1 \, \frac{d}{d\widehat{x}} f(\widehat{x}) \pm c_2 \, \frac{d}{d\widehat{x}} g(\widehat{x}).
\end{align*}
It completes the proof of \eqref{3.3}. Similarly, the other can be verified. 
\end{proof}
Let $F(\widehat{x})$ be a dual valued function such that $\frac{d}{d \widehat{x}} f(\widehat{x}) = F(\widehat{x})$. Then, $F(\widehat{x})$ will be called dual anti-derivative of $f(\widehat{x})$, and we represent it by 
\begin{align}
	\int F(\widehat{x}) \, d \widehat{x}.
\end{align}
Since $\frac{d}{d \widehat{x}} f(\widehat{x}) = \frac{\partial}{\partial x_1} [f(x_1) + \varepsilon x_2 f'(x_1)]$, so we evaluate the integral $\int F(\widehat{x}) \, d \widehat{x}$ by $\int F(\widehat{x}) \, d {x_1}$. For example, if $F(\widehat{x}) = (\widehat{x})^k$, then 
\begin{align*}
	\int F(\widehat{x}) \, d \widehat{x} &= \int (x_1^k + k \,  \varepsilon x_2 \, x_1^{k - 1}) \, dx_1\\
	& = \frac{x_1^{k + 1}}{k + 1} + \varepsilon x_2 \, {x_1^k} + C\\
	& = \frac{x_1^{k + 1} + (k + 1) \, \varepsilon x_2 \, x_1^k}{k + 1} + C  = \frac{(\widehat{x})^{k + 1}}{k + 1} + C, 
\end{align*}
where $C$ is an integration constant. Similarly, we have the following dual integral formulas
\begin{align*}
	& \int  e^{\widehat{x}} \, d \widehat{x} = e^{\widehat{x}} + C; \\
	&  \int  \sin (\widehat{x})\, d \widehat{x} = - \cos (\widehat{x}) + C;\\
	& \int  \cos (\widehat{x}) \, d \widehat{x} =  \sin (\widehat{x}) + C; \\
	&  \int  \tan (\widehat{x}) \, d \widehat{x} = \log \left(\sec (\widehat{x})\right) + C;\\
	& \int \cot (\widehat{x}) \, d \widehat{x} = - \log \left(\csc (\widehat{x})\right) + C;\\
	& \int  \sec (\widehat{x}) \, d \widehat{x} = \log (\sec (\widehat{x}) + \tan (\widehat{x})) + C;\\
	& \int  \csc (\widehat{x}) \, d \widehat{x} = \log  (\csc (\widehat{x}) - \cot (\widehat{x})) + C.
\end{align*}
For dual valued functions $f(\widehat{x})$ and $g(\widehat{x})$, the following algebraic properties for a dual integral operator can be verified easily
\begin{align*}
	\int (c_1 \, f(\widehat{x}) \pm c_2 \, g(\widehat{x})) \, d \widehat{x} & = c_1   \int f(\widehat{x}) \, d \widehat{x} \, \pm  \, c_2   \int g(\widehat{x}) \, d \widehat{x}, \quad c_1, c_2 \in \mathbb{R};
	\\[5pt]
	\int ( f(\widehat{x}) \cdot  g(\widehat{x})) d\widehat{x} & =    f(\widehat{x}) \int g(\widehat{x}) \, d \widehat{x} - \int \left(\frac{d}{d\widehat{x}} f(\widehat{x}) \,  \int g(\widehat{x}) \, d \widehat{x}\right) \, d \widehat{x}.
\end{align*}
Note that in the dual integrals of product, the rule of ILATE as for real integrals will be followed. 
\section{Dual valued gamma function}
Let $a \in \mathbb{R}$ such that $a \ne 0, -1, -2, \dots$. Then, using the integral definition of gamma function given by \cite{d24,emo}
\begin{align}
	\Gamma (a) = \int_{0}^{\infty} e^{-t} \, t^{a - 1} \, dt, \label{3.1}
\end{align}
we define the dual valued gamma function,  for $\widehat{a} = a + \varepsilon b \in \mathbb{D}\mathbb{R}$, as
\begin{align}
	\Gamma_d (\widehat{a}) = \int_{0}^{\infty} e^{-t} \, t^{\widehat{a} - 1} \, dt, \quad a \ne 0, -1, -2, \dots. \label{3.2}
\end{align}
Applying $t^{\varepsilon \, b} = \exp (\varepsilon \, b \, \log t) = 1 + \varepsilon \, b \, \log t$, we rewrite the integral \eqref{3.2} as 
\begin{align}
	\Gamma_d (\widehat{a}) & = \int_{0}^{\infty} e^{-t} \, t^{{a} - 1} \, (1 + \varepsilon \, b \, \log t) \, dt\nonumber\\
	& = \Gamma (a) +   \varepsilon \, b  \int_{0}^{\infty} e^{-t} \, t^{{a} - 1} \,  \log t \, dt, \quad a \ne 0, -1, -2, \dots.
\end{align}
The derivative of gamma function,
\begin{align}
	\Gamma' (a) = \int_{0}^{\infty} e^{-t} \, t^{a - 1} \, \log t \, dt,
\end{align} 
allows to write 
\begin{align}
	\Gamma_d (\widehat{a}) & =   \Gamma (a) +   \varepsilon \, b  \, \Gamma' (a), \quad a \ne 0, -1, -2, \dots.\label{e3.6}
\end{align}
Hence, it can be seen that the dual of gamma function is again a dual number. Since gamma function and its derivative are well-defined and, so dual gamma function can be evaluated. For example, when $a$ is a positive integer, say $k$, then using the relations $\Gamma (k + 1) = k !$ and $\Gamma' (k + 1) = k! \left(-\gamma + \sum_{m = 1}^{k} \frac{1}{m}\right)$, where $\gamma$ is Euler-Mascheroni constant \cite{emo}, we obtain
\begin{align*}
	\Gamma_d (1 + \varepsilon \, b) & = \Gamma (1) + \varepsilon \, b \, \Gamma'(1) = 1 - \varepsilon \, b \, \gamma, \quad \because \Gamma'(1) = -\gamma;
	\\[5pt]
	\Gamma_d (2 + \varepsilon \, b) & = \Gamma (2) + \varepsilon \, b \, \Gamma'(2) = 1 + \varepsilon \, b \, (-\gamma + 1);
	\\[5pt]
	\Gamma_d (3 + \varepsilon \, b) & = \Gamma (3) + \varepsilon \, b \, \Gamma'(3) = 2 !\left[1 + \varepsilon \, b \, \left(-\gamma + 1 + \frac{1}{2}\right)\right].
\end{align*}
In general,
\begin{align*}
	\Gamma_d ((k + 1) + \varepsilon \, b) & = \Gamma (k + 1) + \varepsilon \, b \, \Gamma'(k + 1) \nonumber\\
	& = k ! + \varepsilon \, b \, k ! \, \left(-\gamma  + \sum_{m = 1}^{k}\frac{1}{m}\right) \nonumber\\
	& = k ! \left(1 + \varepsilon \, b \left(-\gamma + \sum_{m = 1}^{k}\frac{1}{m}\right)\right).
\end{align*}
Similarly, for other values of $a$ except $0, -1, -2, \dots$, the value of dual gamma function can be determined. 

We now elaborate few elementary results in terms of dual gamma function. The equation \eqref{e3.6} enables to write 
\begin{align}
	\Gamma_d (\widehat{a} + 1) = \Gamma_d ((a + 1) + \varepsilon \, b) =   \Gamma (a + 1) +   \varepsilon \, b  \, \Gamma' (a + 1).\label{e4.6}
\end{align}
Using the functional equation obeyed by gamma function as $\Gamma (a + 1) = a \, \Gamma (a)$, we write
\begin{align}
	\Gamma_d (\widehat{a} + 1)  & =   a \, \Gamma (a) +   \varepsilon \, b  \, (a \, \Gamma' (a) + \Gamma (a))\nonumber\\
	& = (a + \varepsilon \, b) \, (\Gamma (a) + \varepsilon \, b \, \Gamma'(a)) = \widehat{a} \, 	\Gamma_d (\widehat{a}),\label{2.8}
\end{align}
where we have utilized that $\varepsilon^2 = 0$. Thus the dual gamma function obeyed the functional identity. Employing subsequently identity  \eqref{e4.6} yields
\begin{align}
	\Gamma_d (\widehat{a} + k)   = (\widehat{a} + k - 1) \, (\widehat{a} + k - 2) \cdots (\widehat{a} + 1)  \, \widehat{a} \, 	\Gamma_d (\widehat{a}), \quad k = 0, 1, \dots.\label{a2.8}
\end{align}
The inverse of a dual gamma function can be obtained by using the relation \eqref{a2.8} as
\begin{align}
	\Gamma_d (\widehat{a} + k)  \, [\Gamma_d (\widehat{a})]^{-1} = (\widehat{a} + k - 1) \, (\widehat{a} + k - 2) \cdots (\widehat{a} + 1)  \, \widehat{a}.\label{2.9}
\end{align}
We now use the left-hand side of the above relation \eqref{2.9} to define the dual shifted factorial function. For a dual number $\widehat{a} = a + \varepsilon \, b$, the dual shifted factorial function denoted by $(\widehat{a})_k$ is defined as follows
\begin{align}
	(\widehat{a})_k = \frac{\Gamma_d (\widehat{a} + k)}{\Gamma_d (\widehat{a})} = \begin{cases}
		1, \quad & \text{ if } k = 0,\\
		(\widehat{a} + k - 1) \, (\widehat{a} + k - 2) \cdots (\widehat{a} + 1)  \, \widehat{a}, \quad & \text{ if } k \ge 1. \label{2.10}
	\end{cases}
\end{align}
An immediate consequence of \eqref{2.10} is the derivation of the following related identities. The proofs are elementary and hence we omitted. 
\begin{enumerate}
	\item  $(m + \varepsilon \, b)_k = \frac{(m + k - 1)!}{(m - 1) !} + \varepsilon \, b \frac{d}{dm} \, \frac{(m + k - 1)!}{(m - 1) !}$; 
	\item  $(\widehat{a})_{m + k} = (\widehat{a})_m \, (\widehat{a} + m)_k = (\widehat{a})_k \, (\widehat{a} + k)_m$;
	\item  $(\widehat{a})_{-k} = \frac{(-1)^k}{(1 - \widehat{a})_k}$;
	\item  $(\widehat{a})_{n - k} = \frac{(-1)^k \, (\widehat{a})_n}{(1 - \widehat{a} - n)_k}, \quad 0 \le k \le n$;
	\item   $(-\widehat{a})_k = (-1)^k \, (1 + \widehat{a} - k)_k$;
	\item   $\frac{\left(\frac{\widehat{a}}{2} + 1\right)_k }{\left(\frac{\widehat{a}}{2}\right)_k} = \frac{\widehat{a} + 2k}{\widehat{a}}$.
\end{enumerate} 
Next, we obtain the limiting definition of dual gamma function. To do so, consider $\widehat{a} = a + \varepsilon \, b$ and the following integrals:
\begin{align}
	f(\widehat{a}) & = f (a + \varepsilon \, b) \nonumber\\
	& = \int_{0}^{1} (1 - u)^k \, u^{\widehat{a} - 1} \, du = k ! \left[\widehat{a} (\widehat{a} + 1) \cdots (\widehat{a} + k)\right]^{-1}, \quad a \ne 0, -1, -2, \dots\label{2.12}
\end{align}
and 
\begin{align}
	g(\widehat{a}) & = g (a + \varepsilon \, b)\nonumber\\
	& = \int_{0}^{k} \left(1 - \frac{u}{k}\right)^k \, u^{\widehat{a} - 1} \, du = k ! \, k^{\widehat{a}} \, \left[\widehat{a} (\widehat{a} + 1) \cdots (\widehat{a} + k)\right]^{-1}, \quad a \ne 0, -1, -2, \dots.\label{2.13}
\end{align}
Now, using \eqref{3.2} and \eqref{2.13}, we have
\begin{align}
	&	\Gamma_d (\widehat{a}) - k ! \, k^{\widehat{a}} \, \left[\widehat{a} (\widehat{a} + 1) \cdots (\widehat{a} + k)\right]^{-1}\nonumber\\
	& = \int_{0}^{\infty} e^{-u} \, u^{\widehat{a} - 1} \, du - \int_{0}^{k} \left(1 - \frac{u}{k}\right)^k \, u^{\widehat{a} - 1} \, du\nonumber\\
	& = \int_{0}^{k} \left[e^{-u} - \left(1 - \frac{u}{k}\right)^k \right] \, u^{\widehat{a} - 1} \, du + \int_{k}^{\infty} e^{-u} \, u^{\widehat{a} - 1} \, du.
\end{align}
Denoting $\widehat{a} = a_1 + \varepsilon a_2$, the integral 
\begin{align}
	\int_{0}^{\infty} e^{-u} \, u^{\widehat{a} - 1} \, du = \int_{0}^{\infty} e^{-u} \, u^{a_1} \left(1 + \varepsilon \, a_2 \log (u)\right) \, du < \infty,
\end{align}
then the second integral on the right-hand side is convergent and hence $\int_{k}^{\infty} e^{-u} \, u^{\widehat{a} - 1} \, du \to 0$ as $k \to \infty$. Now, we show 
\begin{align*}
	\int_{0}^{k} \left[e^{-u} - \left(1 - \frac{u}{k}\right)^k \right] \, u^{\widehat{a} - 1} \, du \to 0 \text{ as } k \to \infty.
\end{align*}
Using the inequality, \cite{edr}
\begin{align}
	0 \le e^{-u} - \left(1 - \frac{u}{k}\right)^k \le \frac{u^2 \, e^{-u}}{k}, \quad 0 \le u \le k,
\end{align}
we have
\begin{align*}
	\int_{0}^{k} \left[e^{-u} - \left(1 - \frac{u}{k}\right)^k \right] \, u^{\widehat{a} - 1} \, du \le \frac{1}{k} \, \int_{0}^{k} e^{-u} \, u^{\widehat{a} + 1}  \, du.
\end{align*}
The integral $\int_{0}^{k} e^{-u} \, u^{\widehat{a} + 1}  \, du < + \infty$ as $k \to \infty$. Therefore $\int_{0}^{k} \left[e^{-u} - \left(1 - \frac{u}{k}\right)^k \right] \, u^{\widehat{a} - 1} \, du \to 0$ as $k \to \infty$. Thus, we conclude our finding in the following theorem. 
\begin{theorem}
	Let $\widehat{a} = a + \varepsilon \, b, a, b \in \mathbb{R}$ be a dual number. Then, the limiting definition of dual gamma function is given by
	\begin{align}
		\Gamma_d (\widehat{a}) & = \lim_{k \to \infty} k ! \, k^{\widehat{a}} \, \left[\widehat{a} (\widehat{a} + 1) \cdots (\widehat{a} + k)\right]^{-1}\nonumber\\
		& = \lim_{k \to \infty} k^{\widehat{a}} \, (k - 1) ! \, \frac{k}{\widehat{a} + k} \, [\widehat{a} (\widehat{a} + 1) \cdots (\widehat{a} + k - 1)]^{-1}  = \lim_{k \to \infty} \, \frac{(k - 1) ! \, k^{\widehat{a}}}{(\widehat{a})_k}. 
	\end{align}
\end{theorem}
\section{Dual valued beta function}
Let $a, c \in \mathbb{R}$ such that $a, c \ne 0, -1, -2, \dots$. Then, the integral definition of beta function is given by \cite{emo, edr}
\begin{align}
	\beta (a, c) = \int_{0}^{1} t^{a - 1} \, (1 - t)^{c - 1} \, dt.
\end{align}
Following this, for dual numbers $\widehat{a} = a + \varepsilon \, b, \widehat{c} = c + \varepsilon \, d, a, b, c, d \in \mathbb{R}$, we define the dual beta function as
\begin{align}
	\beta_d (\widehat{a}, \widehat{c}) = \int_{0}^{1} t^{\widehat{a} - 1} \, (1 - t)^{\widehat{c} - 1} \, dt, \quad a, c \ne 0, -1, -2, \dots.
\end{align}
Now, using $t^{\varepsilon \, b}  =  1 + \varepsilon \, b \, \log t$ and $(1 - t)^{\varepsilon \, d} = 1 + \varepsilon \, d \, \log (1 - t)$, we have
\begin{align}
	\beta_d (\widehat{a}, \widehat{c}) & = \int_{0}^{1} t^{{a} - 1} \, (1 - t)^{{c} - 1} \, (1 + \varepsilon \, b \, \log t + \varepsilon \, d \, \log (1 - t)) \, dt\nonumber\\
	& = \int_{0}^{1} t^{{a} - 1} \, (1 - t)^{{c} - 1} \, dt + \varepsilon \, b \int_{0}^{1} t^{{a} - 1} \, (1 - t)^{{c}  - 1} \, \log t \, dt\nonumber\\
	& \quad  + \varepsilon \, d \int_{0}^{1} t^{{a} - 1} \, (1 - t)^{{c} - 1} \, \log (1 - t) \, dt\nonumber\\
	& = \beta (a, c) + \varepsilon \, b \, \frac{d}{d a} \, \beta(a, c)  + \varepsilon \, d \, \frac{d}{d c} \, \beta(a, c).
\end{align}
Which is an expression of dual beta function in terms of beta function and its derivatives. Now, employing the fundamental relation between gamma and beta function $\beta(a, c) = \frac{\Gamma (a) \, \Gamma (c)}{\Gamma (a + c)}$, we get
\begin{align}
	& \beta_d (\widehat{a}, \widehat{c})\nonumber\\
	& = \frac{\Gamma (a) \, \Gamma (c)}{\Gamma (a + c)} + \varepsilon \, b \frac{\Gamma (a + c) \, \Gamma'(a) \, \Gamma (c) - \Gamma (a) \, \Gamma (c) \, \Gamma' (a + c)}{[\Gamma (a + c)]^2}\nonumber\\
	& \quad   + \varepsilon \, d \frac{\Gamma (a + c) \, \Gamma(a) \, \Gamma' (c) - \Gamma (a) \, \Gamma (c) \, \Gamma' (a + c)}{[\Gamma (a + c)]^2}\nonumber\\
	& = \frac{\Gamma (a) \, \Gamma (c) \, \Gamma (a + c) + \varepsilon \, [\Gamma (a + c) (b \, \Gamma'(a) \, \Gamma (c) + d \, \Gamma(a) \, \Gamma' (c)) - (b + d) \Gamma (a) \, \Gamma (c) \, \Gamma' (a + c)]}{[\Gamma (a + c)]^2}\nonumber\\
	& = \frac{\Gamma (a) \, \Gamma (c) + \varepsilon \, (b \, \Gamma'(a) \, \Gamma (c) + d \, \Gamma(a) \, \Gamma' (c))}{\Gamma (a + c) + \varepsilon \, (b + d) \, \Gamma'(a + c)}\nonumber\\
	& = \frac{(\Gamma (a) + \varepsilon \, b \, \Gamma'(a)) \, (\Gamma (c) + \varepsilon \, d \, \Gamma'(c))}{\Gamma (a + c) + \varepsilon \, (b + d) \, \Gamma'(a + c)} = \frac{\Gamma_d (\widehat{a}) \, \Gamma_d (\widehat{c})}{\Gamma_d(\widehat{a} + \widehat{c})},\label{e3.4}
\end{align}
where in the last two steps, we have first applied the identity \eqref{e2.4} and in the second the identity given in \eqref{e3.6}.  Thus, we obtain the fundamental relation between dual gamma and beta functions. With the help of this, we can derive certain functional relations obeyed by the dual beta function as given in the following theorem.
\begin{theorem}
	For dual numbers $\widehat{a} = a + \varepsilon \, b$, $\widehat{c} = c + \varepsilon \, d$, $\widehat{e} = e + \varepsilon \, f$ and $\widehat{g} = g + \varepsilon \, h$, the dual beta function satisfies the following functional relations:
	\begin{enumerate}
		\item $\beta_d (\widehat{a}, \widehat{c} + 1) = \frac{\widehat{c}}{\widehat{a}} \, \beta_d (\widehat{a} + 1, \widehat{c}) = \frac{\widehat{c}}{\widehat{a} + \widehat{c}} \, \beta_d (\widehat{a}, \widehat{c})$.
		
		\item $\beta_d (\widehat{a}, \widehat{c}) \, \beta_d (\widehat{a} + \widehat{c}, \widehat{e}) = \beta_d (\widehat{c}, \widehat{e}) \,  \beta_d (\widehat{c} + \widehat{e}, \widehat{a}) = \beta_d (\widehat{e}, \widehat{a}) \, \beta_d (\widehat{a} + \widehat{e}, \widehat{c})$.
		
		\item $\beta_d (\widehat{a}, \widehat{c}) \, \beta_d (\widehat{a} + \widehat{c}, \widehat{e}) \, \beta_d (\widehat{a} + \widehat{c} + \widehat{e}, \widehat{g}) = \frac{\Gamma_d (\widehat{a}) \, \Gamma_d (\widehat{c}) \, \Gamma_d (\widehat{e}) \, \Gamma_d (\widehat{g})}{\Gamma_d (\widehat{a} + \widehat{c} + \widehat{e} + \widehat{g})}$.
	\end{enumerate}
\end{theorem}
\begin{proof}
	From the relations \eqref{2.8} and  \eqref{e3.4}
	\begin{align}
		\beta_d (\widehat{a}, \widehat{c} + 1) = \frac{\Gamma_d (\widehat{a})  \, \Gamma_d (\widehat{c} + 1)}{\Gamma_d (\widehat{a}  + \widehat{c} + 1)} = \frac{\Gamma_d (\widehat{a})  \, \widehat{c} \, \Gamma_d (\widehat{c})}{(\widehat{a}  + \widehat{c}) \, \Gamma_d (\widehat{a}  + \widehat{c})}.
	\end{align}
	Again
	\begin{align}
		\beta_d (\widehat{a} + 1, \widehat{c}) = \frac{\Gamma_d (\widehat{a} + 1)  \, \Gamma_d (\widehat{c})}{\Gamma_d (\widehat{a}  + \widehat{c} + 1)} = \frac{\widehat{a} \, \Gamma_d (\widehat{a})  \, \Gamma_d (\widehat{c})}{(\widehat{a}  + \widehat{c}) \, \Gamma_d (\widehat{a}  + \widehat{c})}.
	\end{align}
	By combining these two identities and \eqref{e3.4}, one obtains the following ones,
	\begin{align}
		\beta_d (\widehat{a}, \widehat{c} + 1) = \frac{\widehat{c}}{\widehat{a}} \, \beta_d (\widehat{a} + 1, \widehat{c}) = \frac{\widehat{c}}{\widehat{a} + \widehat{c}} \, \beta_d (\widehat{a}, \widehat{c}).
	\end{align}
	It completes the proof of first identity. Similarly, the other two identities can be proved.
\end{proof}
\section{Dual valued generalized hypergeometric function}
Hypergeometric functions are a class of special functions, named 'hypergeometric' because they can be expressed in terms of hypergeometric series. A dual-valued hypergeometric function is obtained by extending either the parameters or the variables to dual numbers. Consequently, this leads to three possible definitions of dual-valued generalized hypergeometric functions. Let $\widehat{a}_1 = a_{11} + \varepsilon a_{12}, \dots, \widehat{a}_p = a_{p1} + \varepsilon a_{p2}$, $\widehat{b}_1 = b_{11} + \varepsilon b_{12}, \dots, \widehat{b}_q = b_{q1} + \varepsilon b_{q2}$ and $\widehat{x} = x_1 + \varepsilon x_2 \in \mathbb{D}\mathbb{R}$. Then, we have the following dual forms of generalized hypergeometric function: 
\begin{align}
	{}_pF_q (a_1, \dots, a_p; b_1, \dots, b_q; \widehat{x}) & = \sum_{k = 0}^{\infty} \frac{(a_1)_k \cdots (a_p)_k}{(b_1)_k \cdots (b_q)_k} \, \frac{(\widehat{x})^k}{k!};\label{3.4}
	\\[5pt]
	{}_pF_q (\widehat{a}_1, \dots, \widehat{a}_p; \widehat{b}_1, \dots, \widehat{b}_q;  {x}) & = \sum_{k = 0}^{\infty} \frac{(\widehat{a}_1)_k \cdots (\widehat{a}_p)_k}{(\widehat{b}_1)_k \cdots (\widehat{b}_q)_k} \, \frac{ {x}^k}{k!};\label{3.5}
	\\[5pt]
		{}_pF_q (\widehat{a}_1, \dots, \widehat{a}_p; \widehat{b}_1, \dots, \widehat{b}_q;  \widehat{x}) & = \sum_{k = 0}^{\infty} \frac{(\widehat{a}_1)_k \cdots (\widehat{a}_p)_k}{(\widehat{b}_1)_k \cdots (\widehat{b}_q)_k} \, \frac{ (\widehat{x})^k}{k!}.\label{3.6}
\end{align} 
If dual parts of each parameter and variable is set to be zero, then all these three forms coincide to the generalized hypergeometric function ${}_pF_q (a_1, \dots, a_p; b_1, \dots, b_q;  {x})$ defined by
\begin{align}
	{}_pF_q (a_1, \dots, a_p; b_1, \dots, b_q;  {x}) & = \sum_{k = 0}^{\infty} \frac{(a_1)_k \cdots (a_p)_k}{(b_1)_k \cdots (b_q)_k} \, \frac{ {x}^k}{k!},
\end{align}
where $(a)_k$ is the shifted factorial given by
\begin{align}
	(a)_k =   \begin{cases}
		1, \quad & \text{ if } k = 0,\\
		(a + k - 1) \, (a + k - 2) \cdots (a + 1)  \, a, \quad & \text{ if } k \ge 1.
	\end{cases}
\end{align}
 Further, the dual forms \eqref{3.4} and \eqref{3.5} can be deduced from \eqref{3.6} and hence we elaborate the dual function \eqref{3.6} in detail. 
\subsection{Regions of convergence}
Let us denote the general term of the series \eqref{3.6} by $A_k$ so, using the ratio test, we get
\begin{align*}
	\left\vert \frac{A_{k + 1}}{A_k}\right \vert & = \left \vert \frac{(\widehat{a}_1 + k) \cdots (\widehat{a}_p + k)}{(\widehat{b}_1 + k) \cdots (\widehat{b}_q + k)}  \, \frac{\widehat{x}}{(k + 1)}\right \vert\nonumber\\
	& =   \frac{( \vert\widehat{a}_1\vert + k) \cdots (\vert\widehat{a}_p\vert + k)}{(\vert\widehat{b}_1\vert + k) \cdots (\vert\widehat{b}_q\vert + k)}  \, \frac{\vert\widehat{x}\vert}{(k + 1)}\\
	& = \frac{\left(1 +  \frac{\vert\widehat{a}_1\vert}{k}\right) \cdots \left(1 +  \frac{\vert\widehat{a}_p\vert}{k}\right)}{\left(1 +  \frac{\vert\widehat{b}_1\vert}{k}\right) \cdots \left(1 +  \frac{\vert\widehat{b}_q\vert}{k}\right)}  \, \frac{\vert\widehat{x}\vert}{\left(1 +  \frac{1}{k}\right)} \, k^{p - q - 1}.
\end{align*}
Thus, we conclude that the dual hypergeometric function \eqref{3.6} converges or diverges in the following regions
	\begin{enumerate}
	\item If $p \le q$, the dual function converges absolutely for all finite $\vert\widehat{x} \vert$.
	\item If $p = q + 1$, function converges for $\vert \widehat{x} \vert < 1$ and diverges for $\vert \widehat{x} \vert > 1$.
	\item If $p > q + 1$, the dual function diverges for all $\widehat{x} \ne 0$ until it terminates.
\end{enumerate}
If $p = q + 1$ and $\vert \widehat{x} \vert = 1$, then the ratio test fails and we have the following theorem:
\begin{theorem}
	Let $\widehat{a}_1, \dots, \widehat{a}_p, \widehat{b}_1, \dots, \widehat{b}_q \in \mathbb{D} \mathbb{R}$  such that $\Re {(\widehat{b}_1 + \cdots + \widehat{b}_q - (\widehat{a}_1 +  \cdots + \widehat{a}_p))} > 0$. Then, the generalized dual hypergeometric function converges absolutely for $p = q + 1$ and $\vert \widehat{x} \vert = 1$. 
\end{theorem}
\begin{proof}
	Let $\Re ({\widehat{b}_1 + \cdots + \widehat{b}_q - (\widehat{a}_1 +  \cdots + \widehat{a}_p)}) = 2 \delta > 0$. Then 
	\begin{align}
		k^{1 + \delta} \, \frac{ (\widehat{a}_1)_k \cdots  (\widehat{a}_p)_k}{(\widehat{b}_1)_k \cdots (\widehat{b}_q)_k \, k!} & = 	\frac{k^{1 + \delta}}{k !} \frac{(k - 1)! \, k^{\widehat{a}_1} \, k^{-\widehat{a}_1} \, (\widehat{a}_1)_k}{(k - 1)!} \cdots \frac{(k - 1)! \, k^{\widehat{a}_p} \, k^{-\widehat{a}_p} \, (\widehat{a}_p)_k}{(k - 1)!} \nonumber\\
		& \quad \times  \frac{ (k - 1)!}{(k - 1)! \, k^{\widehat{b}_1} \, k^{-\widehat{b}_1} \, (\widehat{b}_1)_k} \cdots \frac{ (k - 1)!}{(k - 1)! \, k^{\widehat{b}_q} \, k^{-\widehat{b}_q} \, (\widehat{b}_q)_k}\nonumber\\
		& = \frac{ k^{-\widehat{a}_1} \, (\widehat{a}_1)_k}{(k - 1)!} \cdots \frac{ k^{-\widehat{a}_p} \, (\widehat{a}_p)_k}{(k - 1)!}  \, \frac{  (k - 1)!}{k^{-\widehat{b}_1} \, (\widehat{b}_1)_k} \cdots \frac{  (k - 1)!}{k^{-\widehat{b}_q} \, (\widehat{b}_q)_k} \nonumber\\
		& \quad \times [(k - 1)!]^{p - q - 1}  k^{\delta +  \widehat{a}_1 +  \cdots + \widehat{a}_p - (\widehat{b}_1 + \cdots + \widehat{b}_q)}.
	\end{align} 
	Using limiting definition of dual gamma function as $\Gamma_d (\widehat{a}) = \lim_{k \rightarrow \infty} \frac{(k-1)! \ k^{\widehat{a}}}{(\widehat{a})_k}$, we have for $\vert \widehat{x}\vert = 1$ and $p = q + 1$:
	\begin{align}
	&	\lim_{k \rightarrow \infty}	k^{1 + \delta} \, \left \vert \frac{(\widehat{a}_1)_k \cdots (\widehat{a}_p)_k}{(\widehat{b}_1)_k \cdots (\widehat{b}_q)_k} \, \frac{ (\widehat{x})^k}{k!}\right \vert\nonumber\\
		& \le \lim_{k \rightarrow \infty} \left \vert \frac{\Gamma_d (\widehat{b}_1) \cdots \Gamma_d(\widehat{b}_q)}{\Gamma_d (\widehat{a}_1) \cdots \Gamma_d (\widehat{a}_p)} \right \vert \, \left \vert \frac{1}{k^{\widehat{b}_1 + \cdots + \widehat{b}_q - (\widehat{a}_1 +  \cdots + \widehat{a}_p) - \delta}} \right \vert = \left \vert \frac{\Gamma_d (\widehat{b}_1) \cdots \Gamma_d(\widehat{b}_q)}{\Gamma_d (\widehat{a}_1) \cdots \Gamma_d (\widehat{a}_p)} \right \vert \cdot  0 = 0,
	\end{align}
where we have utilized that $\vert \widehat{x}\vert = \vert \Re (\widehat{x})\vert$ and equation \eqref{2.4}, to obtain 
\begin{align}
	\vert k^{\widehat{b}_1 + \cdots + \widehat{b}_q - (\widehat{a}_1 + \cdots + \widehat{a}_p) - \delta} \vert & = \vert \Re \left(k^{\widehat{b}_1 + \cdots + \widehat{b}_q - (\widehat{a}_1 + \cdots + \widehat{a}_p) - \delta}\right)\vert\nonumber\\
	& = k^{\Re(\widehat{b}_1 + \cdots + \widehat{b}_q - (\widehat{a}_1 + \cdots + \widehat{a}_p)) - \delta} = k^{2 \delta - \delta} = k^\delta, \quad \delta > 0,
\end{align}	
and that $\lim_{k \to \infty} k^\delta = \infty$, since $\delta > 0$. 
	Therefore, by comparison theorem of numerical series of positive numbers we get the absolute convergence of \eqref{3.6}.
\end{proof}
Next, using the dual differential operator $\frac{d}{d \widehat{x}}$, we deduce
\begin{align}
	& \frac{d}{d \widehat{x}} {}_pF_q (\widehat{a}_1, \dots, \widehat{a}_p; \widehat{b}_1, \dots, \widehat{b}_q;  \widehat{x})\nonumber\\
	& = \frac{\widehat{a}_1 \cdots \widehat{a}_p}{\widehat{b}_1 \cdots \widehat{b}_q} \, {}_pF_q (\widehat{a}_1 + 1, \dots, \widehat{a}_p + 1; \widehat{b}_1 + 1, \dots, \widehat{b}_q + 1;  \widehat{x}).
\end{align}
Which, in general, can be written as
\begin{align}
	& \left(\frac{d}{d \widehat{x}}\right)^r {}_pF_q (\widehat{a}_1, \dots, \widehat{a}_p; \widehat{b}_1, \dots, \widehat{b}_q;  \widehat{x})\nonumber\\
	& = \frac{(\widehat{a}_1)_r \cdots (\widehat{a}_p)_r}{(\widehat{b}_1)_r \cdots (\widehat{b}_q)_r} \, {}_pF_q (\widehat{a}_1 + r, \dots, \widehat{a}_p + r; \widehat{b}_1 + r, \dots, \widehat{b}_q + r;  \widehat{x}).
\end{align}
An intensive use of the operator $\frac{d}{d \widehat{x}}$ yield the dual differential equation obeyed by the dual valued generalized hypergeometric function. Let $\widehat{\theta}$ be the dual differential operator defined by $\widehat{\theta} = \widehat{x} \, \frac{d}{d \widehat{x}}$. Then one can evaluate
\begin{align}
	& [\widehat{\theta} (\widehat{\theta} + \widehat{b}_1 - 1) \cdots (\widehat{\theta} + \widehat{b}_q - 1)] \, {}_pF_q (\widehat{a}_1, \dots, \widehat{a}_p; \widehat{b}_1, \dots, \widehat{b}_q;  \widehat{x})\nonumber\\
	& = \sum_{k = 0}^{\infty} \frac{(\widehat{a}_1)_{k + 1} \cdots (\widehat{a}_p)_{k + 1}}{(\widehat{b}_1)_k \cdots (\widehat{b}_q)_k} \, \frac{ (\widehat{x})^{k + 1}}{k!}. \label{4.6}
\end{align}
Also
\begin{align}
	& [ (\widehat{\theta} + \widehat{a}_1) \cdots (\widehat{\theta} + \widehat{a}_p)] \, {}_pF_q (\widehat{a}_1, \dots, \widehat{a}_p; \widehat{b}_1, \dots, \widehat{b}_q;  \widehat{x})\nonumber\\
	& = \sum_{k = 0}^{\infty} \frac{(\widehat{a}_1)_{k + 1} \cdots (\widehat{a}_p)_{k + 1}}{(\widehat{b}_1)_k \cdots (\widehat{b}_q)_k} \, \frac{ (\widehat{x})^k}{k!}. \label{4.7}
\end{align}
Thus from \eqref{4.6} and \eqref{4.7}
\begin{align}
	& [ \widehat{\theta} (\widehat{\theta} + \widehat{b}_1 - 1) \cdots (\widehat{\theta} + \widehat{b}_q - 1) - \widehat{x} \, (\widehat{\theta} + \widehat{a}_1) \cdots (\widehat{\theta} + \widehat{a}_p)] {}_pF_q (\widehat{a}_1, \dots, \widehat{a}_p; \widehat{b}_1, \dots, \widehat{b}_q;  \widehat{x}) = 0.
\end{align}
Furthermore, the contiguous relations can be derived using the dual differential operator. When the parameters of the generalized dual hypergeometric function are incremented or decremented by unity, the resulting functions are known as contiguous functions, and the relations between them are referred to as contiguous function relations. Rainville \cite{edr45} established these relations for the classical generalized hypergeometric function. Following this approach, we extend the results to the generalized dual-valued hypergeometric function. 

There are $(p+q-1)$ relations that connect either $F$, $F(\widehat{a}_1+)$ and $F(\widehat{a}_k+)$ for $k= 2, \dots, p$, or  $F$, $F(\widehat{a}_1+)$ and $F(\widehat{b}_j-)$ for $j= 1, \dots, q$. These relations are given by
\begin{align}
	&(\widehat{a}_1-\widehat{a}_k)F = \widehat{a}_1F(\widehat{a}_1+) - \widehat{a}_kF(\widehat{a}_k+), \qquad k = 2,  \dots, p, \label{eq6.2} \\
	&(\widehat{a}_1-\widehat{b}_k+1)F = \widehat{a}_1F(\widehat{a}_1+) - (\widehat{b}_k-1)F(\widehat{b}_k-), \qquad     k = 1, \dots, q.\label{eq6.3}
\end{align}
Also, there are $(p+1)$ relations each containing $F$ and $(q+1)$ of its contiguous dual functions given by
\begin{align}
	& \widehat{a}_1F  = \widehat{a}_1F(\widehat{a}_1+) - \widehat{x} \sum_{j=1}^{q} U_j F(\widehat{b}_j+), \qquad p < q, \label{eq6.4}\\
	&(\widehat{a}_1 + \widehat{x})F  = \widehat{a}_1F(\widehat{a}_1+) - \widehat{x} \sum_{j=1}^{q} U_j F(\widehat{b}_j+), \qquad p = q,\label{eq6.5} \\
	&\left[(1-\widehat{x})\widehat{a}_1 + \left(\sum_{j=1}^{p} \widehat{a}_j - \sum_{j=1}^{q} \widehat{b}_j\right)\widehat{x}\right]F \nonumber\\
	&\qquad = (1-\widehat{x})\widehat{a}_1 \, F(\widehat{a}_1+) - \widehat{x} \sum_{j=1}^{q} U_j F(\widehat{b}_j+), \quad p = q+1, \label{eq6.6}
\end{align} 
where $U_j = \displaystyle\prod_{s=1}^{p} \ \frac{\widehat{a}_s-\widehat{b}_j}{\widehat{b}_j} \prod_{s=1, s\ne j}^{q} \frac{1}{ \widehat{b}_s-\widehat{b}_j}$.

The remaining $p$ relations are given by 
\begin{align}
	F &= F(\widehat{a}_k-) + \widehat{x} \sum_{j=1}^{q} W_{j,k} \, F(\widehat{b}_j+); \quad p\leq q; \ k = 1,  \dots, p,
	\label{eq6.8} \\
	(1-\widehat{x})F & = F(\widehat{a}_k-) + \widehat{x} \sum_{j=1}^{q} W_{j,k} \, F(\widehat{b}_j+); \quad p = q +1; \ k = 1,  \dots, p, \label{eq6.9}
\end{align}
where $W_{j,k} = \frac{1}{\widehat{b}_j} \displaystyle
\prod_{s=1, s\ne j}^{q} \frac{1}{\widehat{b}_s-\widehat{b}_j}  \prod_{s=1, s\ne k }^{p}(\widehat{a}_s-\widehat{b}_j)$ and 
\begin{align*}
F & = 	{}_pF_q (\widehat{a}_1, \dots, \widehat{a}_p; \widehat{b}_1, \dots, \widehat{b}_q;  \widehat{x});
\\[5pt]
F(\widehat{a}_i \pm) & = {}_pF_q (\widehat{a}_1, \dots, \widehat{a}_{i - 1}, \widehat{a}_i \pm 1, \widehat{a}_{i + 1}, \dots, \widehat{a}_p; \widehat{b}_1, \dots, \widehat{b}_q;  \widehat{x}), \quad 1 \le i \le p;
\\[5pt]
F(\widehat{b}_j \pm) & = {}_pF_q (\widehat{a}_1, \dots, \widehat{a}_p; \widehat{b}_1, \dots, \widehat{b}_{j - 1}, \widehat{b}_j \pm 1, \widehat{b}_{j + 1}, \dots, \widehat{b}_q;  \widehat{x}), \quad 1 \le j \le q.
\end{align*}
Another prominent way to represent a hypergeometric function is through its integral form. In the following theorem, we explore the integral representations of the dual-valued generalized hypergeometric function. 
\begin{theorem}\label{t4.1}
Let $\widehat{a}_1, \dots, \widehat{a}_p, \widehat{b}_1, \dots, \widehat{b}_q$ and $\widehat{x}$ be dual numbers such that $\Re(\widehat{a}_1), \Re(\widehat{b}_1), \Re(\widehat{b}_1 - \widehat{a}_1) > 0$. Then, the generalized dual hypergeometric function can be represented in the following integral forms:
\begin{align}
& {}_pF_q (\widehat{a}_1, \dots, \widehat{a}_p; \widehat{b}_1, \dots, \widehat{b}_q;  \widehat{x})\nonumber\\
& = \frac{\Gamma_d(\widehat{b}_1)}{\Gamma_d(\widehat{a}_1) \, \Gamma_d(\widehat{b}_1 - \widehat{a}_1)} \int_{0}^{1} u^{\widehat{a}_1 - 1} \, (1 - u)^{\widehat{b}_1 - \widehat{a}_1 - 1} \, {}_pF_q (\widehat{a}_2, \dots, \widehat{a}_p; \widehat{b}_2, \dots, \widehat{b}_q; u \,  \widehat{x}) \, d u;\label{4.16}
\\[5pt]
& = \frac{\Gamma_d(\widehat{b}_1)}{\Gamma_d(\widehat{a}_1) \, \Gamma_d(\widehat{b}_1 - \widehat{a}_1)} \int_{0}^{\infty} u^{\widehat{a}_1 - 1} \, (1 + u)^{-\widehat{a}_1 - \widehat{b}_1 } \, {}_pF_q \left(\widehat{a}_2, \dots, \widehat{a}_p; \widehat{b}_2, \dots, \widehat{b}_q;   \frac{u}{1 + u}\widehat{x}\right) \, d u;\label{4.17}
\\[5pt]
& = \frac{\Gamma_d(\widehat{b}_1)}{\Gamma_d(\widehat{a}_1) \, \Gamma_d(\widehat{b}_1 - \widehat{a}_1)} \, b^{\widehat{a}_1} \int_{0}^{\infty} u^{\widehat{a}_1 - 1} \, (1 + bu)^{-\widehat{a}_1 - \widehat{b}_1 } \, {}_pF_q \left(\widehat{a}_2, \dots, \widehat{a}_p; \widehat{b}_2, \dots, \widehat{b}_q;   \frac{bu}{1 + bu}\widehat{x}\right) \, d u.\label{4.18}
\end{align}
\end{theorem} 
\begin{proof}
The definition of the dual shifted factorial in terms of dual gamma function yields
	\begin{align}
		\frac{(\widehat{a}_1)_k}{(\widehat{b}_1)_k} = \frac{\Gamma_d(\widehat{b}_1)}{\Gamma_d(\widehat{a}_1)} \, \frac{\Gamma_d(\widehat{a}_1 + k)}{\Gamma_d(\widehat{b}_1 + k)} = \frac{\Gamma_d(\widehat{b}_1)}{\Gamma_d(\widehat{a}_1) \, \Gamma_d(\widehat{b}_1 - \widehat{a}_1)} \, \beta_d (\widehat{a}_1 + k, \widehat{b}_1 - \widehat{a}_1).
	\end{align}
Now, using the integral definition of dual beta function given in \eqref{2.9}, we write
\begin{align}
	\frac{(\widehat{a}_1)_k}{(\widehat{b}_1)_k} =  \frac{\Gamma_d(\widehat{b}_1)}{\Gamma_d(\widehat{a}_1) \, \Gamma_d(\widehat{b}_1 - \widehat{a}_1)} \,  \int_{0}^{1} u^{\widehat{a}_1 + k - 1} \, (1 - u)^{\widehat{b}_1 - \widehat{a}_1 - 1} \, d u.
\end{align}
Implementing this into the definition of generalized dual hypergeometric function gives the required integral \eqref{4.16}. Similarly using the other integrals of beta function, we can obtain the integrals from \eqref{4.17} and \eqref{4.18}. 
\end{proof}
\section{Special Cases}
Special cases of the generalized hypergeometric function give rise to various elementary classical functions. Here, we provide a list of dual functions derived from specific values of the generalized dual hypergeometric function. Utilizing the established properties of this generalized dual function, we can systematically deduce corresponding properties for its special cases. This allows  to extend fundamental results to specific dual functions, enhancing our understanding of their behavior and relationships.

If $p = 0 = q$,  that is, no numerator and denominator parameters are involved, then we get
\begin{align*}
{}_0F_0 (-; -;  \widehat{x}) = \sum_{k = 0}^{\infty} \frac{(\widehat{x})^k}{k !} = e^{\widehat{x}} = e^{x_1} (1 + \varepsilon \, x_2). 
\end{align*}
In case $p = 1, q = 0$, we obtain the dual analogue of binomial theorem as 
\begin{align}
{}_1F_0 (\widehat{a}_1; -;  \widehat{x})  = \sum_{k = 0}^{\infty} \frac{(\widehat{a}_1)_k \, (\widehat{x})^k}{k !} & = (1 - \widehat{x})^{-\widehat{a}_1}. \label{5.1}
\end{align}
Similarly, we can deduce the following dual valued functions for distinct values of the numerator and denominator parameters. 
\begin{align}
	{}_1F_1 (\widehat{a}_1; \widehat{b}_1;  \widehat{x}) & = \sum_{k = 0}^{\infty} \frac{(\widehat{a}_1)_k \, (\widehat{x})^k}{(\widehat{b}_1)_k \, k !};\label{5.2}
\\[5pt]
	{}_2F_1 (\widehat{a}_1, \widehat{a}_2; \widehat{b}_1;  \widehat{x}) & = \sum_{k = 0}^{\infty} \frac{(\widehat{a}_1)_k \, (\widehat{a}_2)_k \, (\widehat{x})^k}{(\widehat{b}_1)_k \, k !};\label{5.3}
\\[5pt]
	\widehat{x} \, {}_2F_1 \left(\frac{1}{2}, \frac{1}{2}; \frac{3}{2};  (\widehat{x})^2\right) & = \arcsin \widehat{x};
\\[5pt]
	\widehat{x} \, {}_2F_1 \left(\frac{1}{2}, 1; \frac{3}{2};  -(\widehat{x})^2\right) & = \arctan \widehat{x};
\\[5pt]
	\widehat{x} \, {}_2F_1 \left(1, 1; 2;  - \widehat{x} \right) = \widehat{x} - \frac{(\widehat{x})^2}{2!} + \frac{(\widehat{x})^3}{3!} - \cdots & = \log (1 + \widehat{x});
\\[5pt]
	2 \, \widehat{x} \, {}_2F_1 \left(\frac{1}{2}, 1; \frac{3}{2};  (\widehat{x})^2\right) & = \log \left(\frac{1 + \widehat{x}}{1 - \widehat{x}}\right);
\\[5pt]
{}_2F_1 \left(-n, 1; 1;  - \widehat{x} \right) & =  (1 + \widehat{x})^n.
\end{align}
The functions defined in \eqref{5.2} and \eqref{5.3} will be the dual versions of confluent and Gauss hypergeometric functions given by
\begin{align}
	{}_1F_1 ( {a}_1;  {b}_1;   {x}) & = \sum_{k = 0}^{\infty} \frac{({a}_1)_k \,  {x}^k}{({b}_1)_k \, k !};
\\[5pt]
{}_2F_1 ({a}_1, {a}_2;  {b}_1;   {x}) & = \sum_{k = 0}^{\infty} \frac{({a}_1)_k \, ({a}_2)_k \, {x}^k}{({b}_1)_k \, k !}.
\end{align}
Since these two functions have a wide range of applications in physics, mathematics, engineering, and probability, we study their dual versions in detail. 

The dual valued confluent hypergeometric function obtained as particular case of generalized dual hypergeometric function converges absolutely for $\vert \widehat{x}\vert < 1$. It obeys the dual differential equation
\begin{align}
[ \widehat{\theta} (\widehat{\theta} + \widehat{b}_1 - 1)   - \widehat{x} (\widehat{\theta} + \widehat{a}_1)  ] {}_1F_1 (\widehat{a}_1; \widehat{b}_1;  \widehat{x}) = 0.
\end{align}
Following differential formulas are satisfied by the dual confluent function
\begin{align}
		& \left(\frac{d}{d \widehat{x}}\right)^r {}_1F_1 (\widehat{a}_1; \widehat{b}_1;  \widehat{x}) = \frac{(\widehat{a}_1)_r }{(\widehat{b}_1)_r } \, {}_1F_1 (\widehat{a}_1 + r; \widehat{b}_1 + r;  \widehat{x});
		\\[5pt]
		&\left(\frac{d}{d \widehat{x}}\right)^r \left[(\widehat{x})^{\widehat{a}_1 + r - 1} {}_1F_1 (\widehat{a}_1; \widehat{b}_1;  \widehat{x})\right]  = (\widehat{a}_1)_r \, (\widehat{x})^{\widehat{a}_1 - 1} \, {}_1F_1 (\widehat{a}_1 + r; \widehat{b}_1;  \widehat{x});
		\\[5pt]
	&	\left(\frac{d}{d \widehat{x}}\right)^r \left[(\widehat{x})^{\widehat{b}_1 - 1} \, {}_1F_1 (\widehat{a}_1; \widehat{b}_1;  \widehat{x})\right] = (-1)^r \, (1 - \widehat{b}_1)_r \, (\widehat{x})^{\widehat{b}_1 - 1 - r} \, {}_1F_1 (\widehat{a}_1; \widehat{b}_1 - r;  \widehat{x});
		\\[5pt]
	&	\left(\frac{d}{d \widehat{x}}\right)^r \left[e^{-\widehat{x}} \, {}_1F_1 (\widehat{a}_1; \widehat{b}_1;  \widehat{x})\right] = \frac{(-1)^r \, (\widehat{b}_1 - \widehat{a}_1)_r}{(\widehat{b}_1)_r} \, {}_1F_1 (\widehat{b}_1 - \widehat{a}_1 + r; \ \widehat{b}_1 + r; \ -\widehat{x});
		\\[5pt]
		&\left(\frac{d}{d \widehat{x}}\right)^r \left[e^{-\widehat{x}} \, (\widehat{x})^{\widehat{b}_1 - \widehat{a}_1 + r - 1} \, {}_1F_1 (\widehat{a}_1; \widehat{b}_1;  \widehat{x})\right]  = (\widehat{b}_1 - \widehat{a}_1)_r \, (\widehat{x})^{\widehat{b}_1 - \widehat{a}_1 - 1} \,  {}_1F_1 (\widehat{b}_1 - \widehat{a}_1 + r; \ \widehat{b}_1; \ -\widehat{x});
		\\[5pt]
	&	\left(\frac{d}{d \widehat{x}}\right)^r \left[e^{-\widehat{x}} \, (\widehat{x})^{\widehat{b}_1  - 1} \, {}_1F_1 (\widehat{a}_1; \widehat{b}_1;  \widehat{x})\right]  = (-1)^r \,  (1 - \widehat{b}_1)_r \, (\widehat{x})^{\widehat{b}_1 - r - 1} {}_1F_1 (\widehat{b}_1 - \widehat{a}_1;  \widehat{b}_1 - r;  -\widehat{x}).
\end{align}
The proofs of these dual differential formulas are very elementary and hence we omit them. The dual valued confluent hypergeometric function satisfies the following contiguous relations:
\begin{align}
	(\widehat{a}_1 - \widehat{b}_1 + 1) \, {}_1F_1(\widehat{a}_1; \widehat{b}_1; \widehat{x}) & = \widehat{a}_1 \, {}_1F_1(\widehat{a}_1 + 1; \widehat{b}_1; \widehat{x}) - (\widehat{b}_1 - 1) \, {}_1F_1(\widehat{a}_1; \widehat{b}_1 - 1; \widehat{x});
	\\[5pt]
	\widehat{b}_1 \, (\widehat{a}_1 + \widehat{x}) \, {}_1F_1(\widehat{a}_1; \widehat{b}_1; \widehat{x}) & = \widehat{a}_1 \,  \widehat{b}_1 \, {}_1F_1(\widehat{a}_1 + 1; \widehat{b}_1; \widehat{x}) - (\widehat{a}_1 - \widehat{b}_1) \, \widehat{x} \, {}_1F_1(\widehat{a}_1; \widehat{b}_1 + 1; \widehat{x});
	\\[5pt]
	\widehat{b}_1 \, {}_1F_1(\widehat{a}_1; \widehat{b}_1; \widehat{x}) &= \widehat{b}_1 \, {}_1F_1(\widehat{a}_1 - 1; \widehat{b}_1; \widehat{x}) + \widehat{x} \, {}_1F_1(\widehat{a}_1; \widehat{b}_1 + 1; \widehat{x}).
\end{align} 
We have the following integral representations of the dual confluent function. The proofs are similar to the Theorem~\ref{t4.1} and hence are omitted. 
\begin{align}
{}_1F_1 (\widehat{a}_1; \widehat{b}_1;  \widehat{x}) & =  \frac{\Gamma_d(\widehat{b}_1)}{\Gamma_d(\widehat{a}_1) \, \Gamma_d(\widehat{b}_1 - \widehat{a}_1)} \,  \int_{0}^{1} u^{\widehat{a}_1 - 1} \, (1 - u)^{\widehat{b}_1 - \widehat{a}_1 - 1} \, e^{u \, \widehat{x}} \, d u;\nonumber\\
& = \frac{\Gamma_d(\widehat{b}_1)}{\Gamma_d(\widehat{a}_1) \, \Gamma_d(\widehat{b}_1 - \widehat{a}_1)} \, (\widehat{x})^{1 - \widehat{b}_1} \int_{0}^{\widehat{x}} v^{\widehat{a}_1 - 1} \, (\widehat{x} - v)^{\widehat{b}_1 - \widehat{a}_1 - 1} \, e^{v} \, d v;\nonumber\\
& = \frac{\Gamma_d(\widehat{b}_1)}{\Gamma_d(\widehat{a}_1) \, \Gamma_d(\widehat{b}_1 - \widehat{a}_1)} \, e^{\widehat{x}} \int_{0}^{1} v^{\widehat{a}_1 - 1} \, (1 - v)^{\widehat{b}_1 - \widehat{a}_1 - 1} \, e^{- v \widehat{x}} \, d v;\nonumber\\
& = \frac{\Gamma_d(\widehat{b}_1)}{\Gamma_d(\widehat{a}_1) \, \Gamma_d(\widehat{b}_1 - \widehat{a}_1)} \, 2^{1 - \widehat{b}_1} \, e^{\frac{\widehat{x}}{2}} \int_{-1}^{1} \left(\frac{1 - v}{1 + v}\right)^{\widehat{a}_1 - 1} \, (1 + v)^{\widehat{b}_1 -  2} \, e^{- \frac{1}{2} v \widehat{x}} \, d v;\nonumber\\
& = \frac{\Gamma_d(\widehat{b}_1)}{\Gamma_d(\widehat{a}_1) \, \Gamma_d(\widehat{b}_1 - \widehat{a}_1)} \, 2^{1 - \widehat{b}_1} \, e^{\frac{\widehat{x}}{2}} \int_{-1}^{1} \left(\frac{1 + v}{1 - v}\right)^{\widehat{a}_1 - 1} \, (1 - v)^{\widehat{b}_1 -  2} \, e^{ \frac{1}{2} v \widehat{x}} \, d v.
\end{align}
In the following theorem, we prove several dual integral formulas in term of dual confluent hypergeometric function. 
\begin{theorem} 
The integral formulas for dual confluent function given below hold true:
\begin{align}
	\int {}_1F_1 (\widehat{a}_1; \widehat{b}_1;  \widehat{x}) \, d \widehat{x} & = \frac{\widehat{b}_1 - 1}{\widehat{a}_1 - 1} \, {}_1F_1 (\widehat{a}_1 - 1; \widehat{b}_1 - 1;  \widehat{x}) + C;\label{5.10}\\
	\int (\widehat{x})^{\widehat{b}_1 - 1} \, {}_1F_1 (\widehat{a}_1; \widehat{b}_1;  \widehat{x}) \, d \widehat{x} & = \frac{(\widehat{x})^{\widehat{b}_1}}{\widehat{b}_1} \, {}_1F_1 (\widehat{a}_1; \widehat{b}_1 + 1;  \widehat{x}) + C;\label{5.11}\\
	\int (\widehat{x})^{\widehat{a}_1 - 2} \, {}_1F_1 (\widehat{a}_1; \widehat{b}_1;  \widehat{x}) \, d \widehat{x} & = \frac{(\widehat{x})^{\widehat{a}_1 - 1}}{\widehat{a}_1 - 1} \, {}_1F_1 (\widehat{a}_1 - 1; \widehat{b}_1;  \widehat{x}) + C;\label{5.12}\\
	\int e^{-\widehat{x}} \, {}_1F_1 (\widehat{a}_1; \widehat{b}_1;  \widehat{x}) \, d \widehat{x} & = \frac{e^{-\widehat{x}} \,  {\widehat{b}_1 - 1}}{1 + \widehat{a}_1 - \widehat{b}_1} \, {}_1F_1 (\widehat{a}_1; \widehat{b}_1 - 1;  \widehat{x}) + C,\label{5.13}
\end{align}
where $C$ is an integration constant. 
\end{theorem}
\begin{proof}
	From the definition of dual confluent function, we have
\begin{align}
\int {}_1F_1 (\widehat{a}_1; \widehat{b}_1;  \widehat{x}) \, d \widehat{x} 
 & = \int \sum_{k = 0}^{\infty} \frac{(\widehat{a}_1)_k}{(\widehat{b}_1)_k} \, \frac{(\widehat{x})^k}{k !} \, d \widehat{x}\nonumber\\
& = \sum_{k = 0}^{\infty} \frac{(\widehat{a}_1)_k}{(\widehat{b}_1)_k \, k!} \, \int (\widehat{x})^k \, d \widehat{x} \nonumber\\
& = \sum_{k = 0}^{\infty} \frac{(\widehat{a}_1)_k}{(\widehat{b}_1)_k \, (k + 1)!} \,  (\widehat{x})^{k + 1}  + C\nonumber\\
& = \frac{\widehat{b}_1 - 1}{\widehat{a}_1 - 1} \, \sum_{k = 0}^{\infty} \frac{(\widehat{a}_1 - 1)_k}{(\widehat{b}_1 - 1)_k} \, \frac{(\widehat{x})^k}{k !} + C\nonumber\\
& = \frac{\widehat{b}_1 - 1}{\widehat{a}_1 - 1} \, {}_1F_1 (\widehat{a}_1 - 1; \widehat{b}_1 - 1;  \widehat{x}) + C. 
\end{align}
It completes the proof of \eqref{5.10}. Similarly, the other integrals can be derived. 
\end{proof}
Next, we examine the analytic properties of the dual-valued Gauss hypergeometric function. We present its dual differential equation, differential formulas, contiguous relations, integral representations, and transformation formulas. Since the proofs follow a similar approach to those for the generalized and confluent dual hypergeometric functions, we omit them.

 The dual Gauss function converges absolutely for $\vert \widehat{x}\vert < 1$ and for $\vert \widehat{x}\vert = 1$, whenever $\Re (\widehat{b}_1) > \Re(\widehat{a}_1) + \Re(\widehat{a}_2)$. For dual differential operator $\widehat{\theta} = \widehat{x} \, \frac{d}{d \widehat{x}}$, the dual Gauss hypergeometric function satisfies the dual differential equation
\begin{align}
[\widehat{\theta} (\widehat{\theta} + \widehat{b}_1 - 1)   - \widehat{x} (\widehat{\theta} + \widehat{a}_1) \, (\widehat{\theta} + \widehat{a}_2) ] {}_2F_1 (\widehat{a}_1, \widehat{a}_2; \widehat{b}_1;  \widehat{x}) = 0,
\end{align}
or
\begin{align}
	\widehat{x} (1 - \widehat{x}) \, z'' + [\widehat{b}_1 - (1 + \widehat{a}_1 + \widehat{a}_2) \widehat{x}] \, z' - \widehat{a}_1 \, \widehat{a}_2 \, z = 0,
\end{align}
where $z = {}_2F_1 (\widehat{a}_1, \widehat{a}_2; \widehat{b}_1;  \widehat{x})$. 

The dual Gauss function obeys the following differential formulas
\begin{align*}
\left(\frac{d}{d \widehat{x}}\right)^r {}_2F_1 (\widehat{a}_1, \widehat{a}_2; \widehat{b}_1;  \widehat{x}) & = \frac{(\widehat{a}_1)_r \, (\widehat{a}_2)_r }{(\widehat{b}_1)_r } \, {}_2F_1 (\widehat{a}_1 + r, \widehat{a}_2 + r; \widehat{b}_1 + r;  \widehat{x});
\\[5pt]
\left(\frac{d}{d \widehat{x}}\right)^r [(\widehat{x})^{\widehat{a}_1 + r - 1} \, {}_2F_1 (\widehat{a}_1, \widehat{a}_2; \widehat{b}_1;  \widehat{x})] & = (\widehat{a}_1)_r \, (\widehat{x})^{\widehat{a}_1 - 1} \, {}_2F_1 (\widehat{a}_1 + r, \widehat{a}_2; \widehat{b}_1;  \widehat{x});
\\[5pt]
\left(\frac{d}{d \widehat{x}}\right)^r [(\widehat{x})^{\widehat{a}_2 + r - 1} \, {}_2F_1 (\widehat{a}_1, \widehat{a}_2; \widehat{b}_1;  \widehat{x})] & = (\widehat{a}_2)_r \, (\widehat{x})^{\widehat{a}_2 - 1} \, {}_2F_1 (\widehat{a}_1, \widehat{a}_2 + r; \widehat{b}_1;  \widehat{x});
\\[5pt]
\left(\frac{d}{d \widehat{x}}\right)^r [(\widehat{x})^{\widehat{b}_1  - 1} \, {}_2F_1 (\widehat{a}_1, \widehat{a}_2; \widehat{b}_1;  \widehat{x})] & = (\widehat{b}_1 - r)_r \, (\widehat{x})^{\widehat{b}_1 - r - 1} \, {}_2F_1 (\widehat{a}_1, \widehat{a}_2; \widehat{b}_1 - r;  \widehat{x}).
\end{align*}
Further, the dual valued Gauss hypergeometric function satisfies the set of contiguous formulas as given below
\begin{align*}
	(\widehat{a}_1 - \widehat{a}_2) \, F & = \widehat{a}_1 \, F (\widehat{a}_1 +) - \widehat{a}_2 \, F(\widehat{a}_2 +);
	\\[5pt]
	(\widehat{a}_1 - \widehat{b}_1 + 1) \, F & = \widehat{a}_1 \, F (\widehat{a}_1 +) - (\widehat{b}_1 - 1) \, F(\widehat{b}_1 -);
	\\[5pt]
	[ \widehat{a}_1 + (\widehat{a}_2 - \widehat{b}_1) \widehat{x}] \, F & = \widehat{a}_1 (1 - \widehat{x}) \, F(\widehat{a}_1 +) - \widehat{x} \, \frac{(\widehat{b}_1 - \widehat{a}_1) \, (\widehat{b}_1 - \widehat{a}_2)}{ \widehat{b}_1} F (\widehat{b}_1 +);
	\\[5pt]
	(1 - \widehat{x}) F & = F  (\widehat{a}_1 -)  - \frac{\widehat{b}_1 - \widehat{a}_2}{\widehat{b}_1} \, \widehat{x} F(\widehat{b}_1 +);
	\\[5pt]
	(1 - \widehat{x}) F & = F  (\widehat{a}_2 -)  - \frac{\widehat{b}_1 - \widehat{a}_1}{\widehat{b}_1} \, \widehat{x} F(\widehat{b}_1 +),
\end{align*}
where $F = {}_2F_1 (\widehat{a}_1, \widehat{a}_2; \widehat{b}_1;  \widehat{x})$, $F(\widehat{a}_1 \pm) = {}_2F_1 (\widehat{a}_1 \pm 1, \widehat{a}_2; \widehat{b}_1;  \widehat{x})$, $F(\widehat{b}_1 \pm) = {}_2F_1 (\widehat{a}_1, \widehat{a}_2; \widehat{b}_1 \pm 1;  \widehat{x})$ and so on. 

Following integral representation of dual Gauss hypergeometric function can be obtained using the integral definition of dual beta function.
\begin{align}
{}_2F_1 (\widehat{a}_1, \widehat{a}_2; \widehat{b}_1;  \widehat{x}) = \frac{\Gamma_d(\widehat{b}_1)}{\Gamma_d(\widehat{a}_1) \, \Gamma_d(\widehat{b}_1 - \widehat{a}_1)} \,  \int_{0}^{1} u^{\widehat{a}_1 - 1} \, (1 - u)^{\widehat{b}_1 - \widehat{a}_1 - 1} \, (1 - u \, \widehat{x})^{-\widehat{a}_2} \, du.
\end{align}
In case $\widehat{x}  = 1$, we have the following formula
\begin{align}
	{}_2F_1 (\widehat{a}_1, \widehat{a}_2; \widehat{b}_1;  1) & = \frac{\Gamma_d(\widehat{b}_1)}{\Gamma_d(\widehat{a}_1) \, \Gamma_d(\widehat{b}_1 - \widehat{a}_1)} \,  \int_{0}^{1} u^{\widehat{a}_1 - 1} \, (1 - u)^{\widehat{b}_1 - \widehat{a}_1 -\widehat{a}_2 - 1}  \, du\nonumber\\
	& = \frac{\Gamma_d(\widehat{b}_1)}{\Gamma_d(\widehat{a}_1) \, \Gamma_d(\widehat{b}_1 - \widehat{a}_1)} \,  \frac{\Gamma_d(\widehat{a}_1) \, \Gamma_d (\widehat{b}_1 - \widehat{a}_1 -\widehat{a}_2)}{ \Gamma_d(\widehat{b}_1 - \widehat{a}_2)}\nonumber\\
	& =  \frac{\Gamma_d(\widehat{b}_1) \, \Gamma_d (\widehat{b}_1 - \widehat{a}_1 -\widehat{a}_2)}{\Gamma_d(\widehat{b}_1 - \widehat{a}_1) \, \Gamma_d(\widehat{b}_1 - \widehat{a}_2)}.
\end{align}
This formula can be considered as the dual version of summation theorem. 
\begin{theorem}
	If $\vert \widehat{x} \vert < 1$ and $\left\vert\frac{\widehat{x}}{1 - \widehat{x}}\right \vert < 1$, then following transformation formulas are satisfied by the dual Gauss hypergeometric function
	\begin{align}
		{}_2F_1 (\widehat{a}_1, \widehat{a}_2; \widehat{b}_1; \widehat{x}) = (1 - \widehat{x})^{-\widehat{a}_1} {}_2F_1 \left(\widehat{a}_1, \widehat{b}_1 - \widehat{a}_2; \widehat{b}_1; \frac{-\widehat{x}}{1 - \widehat{x}}\right).\label{5.31}
	\end{align}
\begin{align}
	{}_2F_1 (\widehat{a}_1, \widehat{a}_2; \widehat{b}_1; \widehat{x}) = (1 - \widehat{x})^{\widehat{b}_1 - \widehat{a}_1 - \widehat{a}_2} \, {}_2F_1 (\widehat{b}_1 - \widehat{a}_1, \widehat{b}_1 - \widehat{a}_2; \widehat{b}_1; \widehat{x}).\label{2.5.4}
\end{align}
\end{theorem}
\begin{proof}
	Using the series representations of dual binomial theorem given in \eqref{5.1} and dual Gauss hypergeometric function, we have
\begin{align}
&	(1 - \widehat{x})^{-\widehat{a}_1} {}_2F_1 \left(\widehat{a}_1, \widehat{b}_1 - \widehat{a}_2; \widehat{b}_1; \frac{-\widehat{x}}{1 - \widehat{x}}\right)\nonumber\\
 & = \sum_{m = 0}^{\infty} \frac{(-1)^m \, (\widehat{a}_1)_m \, (\widehat{b}_1 - \widehat{a}_2)_m \, (\widehat{x})^m}{(\widehat{b}_1)_m \, (1 - \widehat{x})^{m + \widehat{a}_1} \, m !}\nonumber\\
	& = \sum_{m, k = 0}^{\infty} \frac{(-1)^m \, (\widehat{a}_1)_m \, (\widehat{b}_1 - \widehat{a}_2)_m \, (\widehat{a}_1 + m)_k \, (\widehat{x})^{m + k} }{(\widehat{b}_1)_m  \, m ! \, k !}.
\end{align}
The identity for dual shifted factorial, $(\widehat{a}_1)_m \, (\widehat{a}_1 + m)_k = (\widehat{a}_1)_{m + k}$, yields
\begin{align}
&	(1 - \widehat{x})^{-\widehat{a}_1} {}_2F_1 \left(\widehat{a}_1, \widehat{b}_1 - \widehat{a}_2; \widehat{b}_1; \frac{-\widehat{x}}{1 - \widehat{x}}\right)\nonumber\\
 & =  \sum_{m, k = 0}^{\infty} \frac{ (\widehat{a}_1)_{m + k} \, (\widehat{b}_1 - \widehat{a}_2)_m \, (\widehat{a}_1 + m)_k \, (-1)^m \, (\widehat{x})^{m + k} }{(\widehat{b}_1)_m  \, m ! \, k !}\nonumber\\
	& = \sum_{k = 0}^{\infty} \sum_{m = 0}^{k} \frac{(-1)^m \, (\widehat{a}_1)_k \, (\widehat{b}_1 - \widehat{a}_2)_m \, (\widehat{x})^k}{(\widehat{b}_1)_m \, m ! \, (k - m)!}\nonumber\\
	& = \sum_{k = 0}^{\infty} \sum_{m = 0}^{k} \frac{(-k)_m \, (\widehat{a}_1)_k \, (\widehat{b}_1 - \widehat{a}_2)_m \, (\widehat{x})^k}{(\widehat{b}_1)_m \, m ! \, k!}\nonumber\\
	& = \sum_{k = 0}^{\infty} \sum_{m = 0}^{k} \frac{(-k)_m \,  (\widehat{b}_1 - \widehat{a}_2)_m}{(\widehat{b}_1)_m \, m ! } \, \frac{(\widehat{a}_1)_k \, (\widehat{x})^k}{k!}\nonumber\\
	& = \sum_{k = 0}^{\infty} {}_2F_1 (-k, \widehat{b}_1 - \widehat{a}_2; \widehat{b}_1; 1) \, \frac{(\widehat{a}_1)_k \, (\widehat{x})^k}{k!}\nonumber\\
	& = \sum_{k = 0}^{\infty} \frac{\Gamma_d (\widehat{b}_1) \, \Gamma_d (\widehat{a}_2 + k)}{\Gamma_d (\widehat{b}_1 + k) \, \Gamma_d (\widehat{a}_2)} \, \frac{(\widehat{a}_1)_k \, (\widehat{x})^k}{k!}\nonumber\\
	& = \sum_{k = 0}^{\infty} \frac{(\widehat{a}_1)_k \, (\widehat{a}_2)_k}{(\widehat{b}_1)_k} \, \frac{z^k}{k !} = {}_2F_1 (\widehat{a}_1, \widehat{a}_2; \widehat{b}_1; \widehat{x}).
\end{align} 
It completes the proof of \eqref{5.31}. Now, using the symmetry of the dual Gauss function, we rewrite the above transformation formula as
\begin{align}
	{}_2F_1 (\widehat{a}_1, \widehat{a}_2; \widehat{b}_1; \widehat{x}) & = (1 - \widehat{x})^{-\widehat{a}_1} {}_2F_1 \left(\widehat{a}_1, \widehat{b}_1 - \widehat{a}_2; \widehat{b}_1; \frac{-\widehat{x}}{1 - \widehat{x}}\right)\nonumber\\
	& = (1 - \widehat{x})^{-\widehat{a}_1} {}_2F_1 \left(\widehat{b}_1 - \widehat{a}_2, \widehat{a}_1; \widehat{b}_1; \frac{-\widehat{x}}{1 - \widehat{x}}\right).
\end{align}
Put $\widehat{x} = \frac{-\widehat{y}}{1 - \widehat{y}}$ and $\frac{1}{1 - \widehat{x}} = 1 - \widehat{y}$, we have 
\begin{align}
	&(1 - \widehat{x})^{\widehat{a}_1 - \widehat{a}_2}  \, {}_2F_1 (\widehat{b}_1 - \widehat{a}_1, \widehat{b}_1 - \widehat{a}_2; \widehat{b}_1; \widehat{x})\nonumber\\
	& = (1 - \widehat{y})^{\widehat{a}_2 - \widehat{b}_1} {}_2F_1 \left(\widehat{b}_1 - \widehat{a}_1, \widehat{b}_1 - \widehat{a}_2; \widehat{b}_1; \frac{-\widehat{y}}{1-\widehat{y}}\right)\nonumber\\
	& = {}_2F_1 (\widehat{b}_1 - \widehat{a}_2, \widehat{a}_1; \widehat{b}_1; \widehat{y}) = {}_2F_1 \left(\widehat{b}_1 - \widehat{a}_2, \widehat{a}_1; \widehat{b}_1; \frac{-\widehat{x}}{1 - \widehat{x}}\right).
\end{align}
Thus
\begin{align}
	&(1 - \widehat{x})^{\widehat{b}_1 - \widehat{a}_1 - \widehat{a}_2} \, {}_2F_1 (\widehat{b}_1 - \widehat{a}_1, \widehat{b}_1 - \widehat{a}_2; \widehat{b}_1; \widehat{x})\nonumber\\
	 &= (1 - \widehat{x})^{-\widehat{a}_1} (1 - \widehat{x})^{\widehat{b}_1 - \widehat{a}_2} {}_2F_1 (\widehat{b}_1 - \widehat{a}_2, \widehat{b}_1 - \widehat{a}_1; \widehat{b}_1; \widehat{x})\nonumber\\
	& = (1 - \widehat{x})^{-\widehat{a}_1} {}_2F_1 \left(\widehat{b}_1 - \widehat{a}_2, \widehat{a}_1; \widehat{b}_1; \frac{-\widehat{x}}{1 - \widehat{x}}\right) \nonumber\\
	& = {}_2F_1 (\widehat{a}_1, \widehat{a}_2; \widehat{b}_1; \widehat{x}).
\end{align}
It completes the proof.
\end{proof}
\section{Conclusion}
Dual numbers were introduced by Clifford in the 19th century, they extend real numbers by introducing an infinitesimal unit $\varepsilon$, where $\varepsilon^2 = 0$. This unique structure allows dual numbers to capture small perturbations in mathematical models, making them particularly useful in computational methods and theoretical physics. Over the years, extensions and variations of dual numbers including hyper-dual numbers, dual-complex numbers,  split-complex numbers and dual matrices, can be seen in \cite{cwk,pvd,vmt}, found applications in distinct mathematical and scientific domains. One can visit these papers and references therein \cite{ar,ja,wb,cs15,cs17,dhhk,drh,fa,isf,gw,msc,ovr,pvfa,pv,pvd,eac,rb,reaw,vt,yi}.

This study of dual hypergeometric functions opens up new avenues for exploring special functions, offering a deeper understanding of their properties and applications. By extending this framework to other hypercomplex number systems, such as quaternions, dual quaternions, and split complex numbers, we can further enrich the theory of hypergeometric functions. Given that both dual numbers and hypergeometric functions have significant applications in mathematics and physics, there is a strong potential for discovering new applications in mathematical physics, particularly in areas involving complex systems, quantum mechanics, and theoretical physics.

\end{document}